\documentclass[12pt,twoside,reqno]{amsart}
\usepackage{amsmath}
\usepackage{amsthm}
\usepackage{amssymb}
\usepackage{amsrefs}
\usepackage{mathrsfs}
\usepackage{mathtools}
\usepackage[usenames]{color}
\usepackage[utf8]{inputenc}
\usepackage{latexsym}
\usepackage{yhmath}
\usepackage{graphicx}
\usepackage{ifthen}
\usepackage{pgf,tikz}
\usepackage{mathrsfs}
\usepackage[colorlinks=true,citecolor=blue]{hyperref}
\usetikzlibrary{arrows}
\usepackage{setspace}
\newtheorem{thm}{}[section]
\newtheorem{theorem}[thm]{Theorem}
\newtheorem{cor}[thm]{Corollary}

\newtheorem{lem}[thm]{Lemma}
\newtheorem{prop}[thm]{Proposition}
\newtheorem{defn}[thm]{Definition}
\newtheorem{remark}[thm]{Remark}

\theoremstyle{remark}

\numberwithin{equation}{section}
\allowdisplaybreaks

\newcommand{\Env}[2][]{%
\ifthenelse{ \equal{#1}{} }  
{\ensuremath{#2_{\mathsf{c}}}}  
{\ensuremath{#2_{\mathsf{c},#1}}}
}

\newcommand{\cN}{\mathcal{NN}}
\newcommand{\N}{\mathbb{N}}
\newcommand{\R}{\mathbb{R}}

\newcommand{\X}{\mathbb{X}}
\newcommand{\Y}{\mathbb{Y}}

\newcommand{\D}{\mathcal{D}}
\newcommand{\Realization}{\mathrm{R}}
\newcommand{\norm}[1][\cdot]{ \Vert #1 \Vert }
\def\ba{\begin{eqnarray*}}
	\def\ea{\end{eqnarray*}}

\def\bee{\begin{equation}}
	
	\def\ene{\end{equation}}

\newcommand{\Rbb}{\mathbb{R}}
\newcommand{\Sc}{\mathcal{D}}
\begin{document}
\setcounter{page}{1}

\title[Universality for Non-linear Convex Variational Problems]{Universality for Non-linear Convex Variational Problems}

\author[P.M. Bern\'a]{Pablo M. Bern\'a}
\address{Pablo M. Bern\'a\\
	Departamento de Métodos Cuantitativos, CUNEF Universidad\\ 
	Madrid\\
	28040 Spain} 
\email{pablo.berna@cunef.edu}

\author[A. Falc\'o]{Antonio Falc\'o}
\address{Antonio Falc\'o \\
ESI International Chair@CEU‐UCH \\
Universidad Cardenal Herrera‐CEU, CEU Universities \\ 
San Bartolomé 55, 46115 \\ 
Alfara del Patriarca, Valencia, Spain} 
	\email{afalco@uchceu.es}




\begin{abstract}

	
	This article introduces an innovative mathematical framework designed to tackle non-linear convex variational problems in reflexive Banach spaces. Our approach employs a versatile technique that can handle a broad range of variational problems, including standard ones. To carry out the process effectively, we utilize specialized sets known as radial dictionaries, where these dictionaries encompass diverse data types, such as tensors in Tucker format with bounded rank and Neural Networks with fixed architecture and bounded parameters. The core of our method lies in employing a greedy algorithm through dictionary optimization defined by a multivalued map. Significantly, our analysis shows that the convergence rate achieved by our approach is comparable to the Method of Steepest Descend implemented in a reflexive Banach space, where the convergence rate follows the order of $O(m^{-1})$.
\end{abstract}

\maketitle

\section{Introduction}\label{sec0}

Related to the numerical methods of Partial Differential Equations (PDEs), two particular 
frameworks are attracted in recent years to the scientific computing community: the application
of Tensor Numerical Methods \cites{BSU,FN,FHN,BK} and 
the Deep Neural Networks (DNNs) \cites{Petersen2,Weinan1,Raissi1,Raissi2,Sheng}.
Concerning the use of these methods in engineering and industrial applications,
among the family of tensor-based methods, the Proper Generalized Decomposition (PGD) 
has been used , among others, in surgery simulations \cite{Niroomandi}, design optimization \cites{Ammar,Leygue}, data-driven applications \cite{Gonzalez}, elastodynamic \cite{Boucinha}, as well in structural damage identification \cite{Coelho}.

Here, we propose a generalization of the use of the Proper Generealized Decomposition studied in \cite{FN}. Concretely, we would to propose a mathematical framework that includes, among others, 
the use of tensors formats and DNNs to solve the problem 
\begin{eqnarray}\label{in1}
	\min_{x\in\mathbb X} \mathcal E(x),
\end{eqnarray}
where $\mathcal E:\X \longrightarrow \R$ is an elliptic and differentiable functional defined over a reflexive Banach space $\X.$ 

A model problem of the above framework is the following one (see for example \cite{Dinca} for more details): let $\Omega$ be a bounded domain in $\mathbb{R}^{d}$ $(d \ge 2)$ with 
Lipschitz continuous boundary. For some fixed  $p > 2$, we let
$H^{1,p}_0(\Omega)$, which
is the closure of $C_c^\infty(\Omega)$ (functions in
$C^\infty(\Omega)$ with compact support in $\Omega$) with respect 
the norm
$$
\norm[f] = \left(\sum_{\ell=1}^d
\norm[\partial_{x_\ell}f]_{L^p(\Omega)}^p \right)^{1/p}.
$$
We then introduce the
functional $\mathcal E: H^{1,p}_0(\Omega): \rightarrow \Rbb$ defined by
$$
\mathcal E(u) = \frac{1}{p}\norm[u]^p
- \langle \varphi,u \rangle,
$$
with $\varphi \in H^{1,p}_0(\Omega)^\ast$. Its Fr\'echet differential is
$\mathcal E'(u) = -\Delta_p -\varphi$
where
$$
\Delta_p(u) = \sum_{\ell=1}^d  \frac{\partial}{\partial
	x_{\ell}}\left(\left\vert \frac{\partial u}{\partial
	x_{\ell}}\right\vert^{p-2}  \frac{\partial u}{\partial
	x_{\ell}}\right),
$$
and $\Delta_p$ is known as the $p$-Laplacian.

In this paper, to obtain the solution $u^*$ of \eqref{in1}, we will use the following ``greedy algorithm":
\begin{enumerate}
	\item $u_0=0$;
	\item for $m\geq 1$, $u_m\in u_{m-1}+	\boldsymbol{\nabla}_{u_{m-1}}(\mathcal E;\Sc)$,
\end{enumerate}
where $$	\boldsymbol{\nabla}_u(\mathcal E;\Sc)=\arg\min_{z\in u+\mathcal D}\mathcal E(z),$$
and $\mathcal D\subset\mathbb X$ is a \textit{universal dictionary} (see Definition \ref{dic}).  The objectives that we have in the paper are the next ones.

\begin{enumerate}
	\item[(S1)]  We need to fix conditions on $\mathcal E$ in order to guarantee the 
	existence of $u^* \in \X$ satisfying 
	$$
	\inf_{u \in \X}\mathcal{E}(u) = \min_{u \in \X}\mathcal{E}(u) = \mathcal E(u^*).
	$$
	To this end, we choose the classical framework developed by Akilov and Kantorovich in \cite[Chapter XV]{Kantorovich} (where, for some applied problems, they proved the convergence of the Method of Steepest Descend in a reflexive Banach space).
	\item[(S2)] We introduce and characterize what is a radial (and universal) dictionary $\Sc$ giving some examples extracted from the literature. 
	\item[(S3)] Finally, under the above conditions, we prove that the sequence $\{u_m\}_{m \in \N}$ generated by the greedy algorithm over $\Sc$  ``minimizes'' $\mathcal E$ in the sense that
	$$
	\lim_{m \rightarrow \infty}\mathcal E(u_m) = \mathcal{E}(u^*),
	$$
	and $\lim_{m\rightarrow \infty} u_m = u^*,$
	and also it satisfies
	$$
	\mathcal E(u_m)-\mathcal E(u^*) = O(m^{-1}),
	$$
	which is the same rate of convergence obtained 
	by Akilov and Kantorovich in \cite{Kantorovich} for the Method of Steepest Descend.
\end{enumerate}
The organization of the paper is the following: in the next section we introduce some preliminary definitions and results used along the text related with (S1). In Section~\ref{sec2}, related to the objective (S2), we introduce and characterize radial (universal) dictionaries where, moreover, we construct some examples by using well-known results from the literature. Section~\ref{sec3} is devoted to explain the greedy algorithm by (universal) dictionary optimization and to prove its convergence
and that its rate of convergence is the same that the Method of Steepest Descend implemented in a reflexive Banach space, that ends (S3). Some conclusions and final remarks will be given in Section~\ref{sec4} and, finally, we prove some techincal results in the Appendix A.

\section{Preliminary definitions and results}\label{sec1}
Let $\mathbb X$ be a reflexive Banach space endowed with a norm $\Vert \cdot \Vert$. We denote by $\X^*$ be the dual space of $\X$, endowed as usual with the dual norm $\|\cdot\|_*,$ and $\langle\cdot,\cdot\rangle: \X^*\times \X\rightarrow \R$ is the duality pairing. 
Along this paper, given $A \subset \X$ and $r > 0$, we will denote by
$$
B_{A,r}=\{x \in A: \|x\| \le r\},
$$
and the unit sphere in $\X$ as
$$
S_{\X}=\{x\in \mathbb X:\|x\|=1\},
$$
which is closed and bounded in $\X.$

Here, we are going to attack the following optimization problem
\begin{eqnarray}\label{main}
	\min_{x\in \X} \mathcal E(x),
\end{eqnarray}
where $\mathcal E: \X\rightarrow \mathbb R$ is a Fréchet differentiable 
functional satisfying the following properties:
\begin{itemize}
	\item[(A)] The Fr\'echet derivative $\mathcal E^{\prime}:\X \longrightarrow \X^*$ is locally Lipschitz continuous and
	\item[(B)] $\mathcal E$ is $\X$-elliptic, that is, there are two real numbers $\alpha>0$ 
	and $1 < s \le 2$ such that for every $x,y\in\X$,
	$$\langle \mathcal E'(x)-\mathcal E'(y), x-y\rangle \geq \alpha\Vert x-y\Vert^s.$$
\end{itemize}

\begin{remark}
	Since any Lipschitz continuous function is uniformly continuous, (A) implies that the Fr\'echet derivative
	$\mathcal E^{\prime}$ is also locally uniformly continuous.
\end{remark}

Conditions (A) and (B) are used in \cite[Chapter XV]{Kantorovich} to prove the convergence and also to compute the rate of convergence of the Method of Steepest Descend in reflexive Banach spaces. Moreover, in \cite[Chapter XV  \S 3]{Kantorovich} some applications to elliptic PDEs are
also given.

\begin{remark}
	Taking  the model problem posed in the last section, following \cite{CAN05}, assumption (A)
	is satisfied by the functional $\mathcal E.$ Recall that  if a functional $F:V \longrightarrow W,$
	where $V$ and $W$ are Banach spaces, is Fr\'echet differentiable at
	$v \in V$, then it is also locally Lipschitz continuous at $v \in
	V.$ Since the map $G:H^{1,p}_0(\Omega) \subset L^p(\Omega) \longrightarrow \mathbb{R},$
	given by $G(u)=\norm[u],$ is $\mathcal{C}^2$ for $p \ge 2$ 
	then $\mathcal E$ is also of class $\mathcal{C}^2.$ 
	Hence, $\mathcal E'$ is locally Lipschitz continuous
	in $H^{1,p}_0(\Omega)$ and assumption (B) holds.
	
\end{remark}

Now, we give some results and definitions  that will be useful along 
this paper.

\begin{lem}{\cite[Lemma 2.2]{CAN05}}\label{lemma_assumption}
	Under assumptions (A)-(B), we have
	\begin{itemize}
		\item[(a)] For all $x,y\in \X,$
		\begin{align}
			\mathcal E(x)-\mathcal E(y) \ge \langle \mathcal E'(y),x-y\rangle  +  \frac{\alpha}{s} \Vert
			x-y\Vert^s. \label{eq:J_strong_convexity_consequence}
		\end{align}
		\item[(b)] $\mathcal E$ is strictly convex.
		\item[(c)] $\mathcal E$ is bounded from below and coercive, i.e. $\lim_{\Vert
			x\Vert\rightarrow\infty} \mathcal E(x)=+\infty.$
	\end{itemize}
\end{lem}

\begin{remark}\label{remark_lemma_assumption}
From  \eqref{eq:J_strong_convexity_consequence} we deduce 
\begin{align}
			\mathcal E(x)-\mathcal E(y) \ge \langle \mathcal E'(y),x-y\rangle, \label{eq:J_strong_convexity_consequence0}
		\end{align}
and hence
\begin{align}
			\mathcal E(y)-\mathcal E(x) \le \langle \mathcal E'(y),y-x\rangle, \label{eq:J_strong_convexity_consequence1}
		\end{align}
holds for all $x,y\in \X.$
\end{remark}

The above lemma is the key of the following classical result.

\begin{theorem}{\cite[Theorem~7.4.4]{CIARLET}}\label{th:solution}
	Under assumptions (A)-(B), the problem \eqref{main} admits a
	unique solution $x\in \X$ which is equivalently characterized by
	\begin{align}
		\langle \mathcal{E}'(x),y\rangle  = 0\quad \forall
		y\in \X \label{eq:euler}
	\end{align}
\end{theorem}

As we have explanied in the last section, the idea is to construct the solution of \eqref{main} 
by a ``greedy algorithm" using a set $\mathcal A \subset \X$ with some conditions and the milestone of 
this procedure is the following optimization program
\begin{align}
	\inf_{v\in \mathcal A} \mathcal E(v).\label{eq:inf_problem}
\end{align}
We mention \cite{EKE99} to study some results about the existence of a 
minimizer of these type of problems and some classical definitions and results to justify the existence of solutions
of \eqref{eq:inf_problem} are the following ones.

\begin{defn}
	We recall that a sequence $x_{m}\in \X$ is \emph{weakly convergent}
	if $\lim_{m\rightarrow\infty}\langle \varphi,  x_{m}  \rangle$
	exists for all $\varphi\in \X^{\ast}.$ We say that $\left(
	x_{m}\right) _{m\in\mathbb{N}}$ \emph{converges weakly to} $x\in \X$
	if $\lim_{m\rightarrow\infty} \langle \varphi, x_{m} \rangle
	=\langle \varphi, x\rangle$ for all $\varphi\in \X^{\ast}$. In this
	case, we write $x_{m}\rightharpoonup x$.
\end{defn}

\begin{defn}
	A subset $\mathcal A \subset \X$ is called \emph{weakly closed} if $x_{m}\in \mathcal A$
	and $x_{m}\rightharpoonup x$ implies $x\in \mathcal A.$
\end{defn}

\begin{remark}
	Of course, the condition of weakly closed is stronger than the condition to be close, 
	so if $\mathcal A \subset\X$ is weakly closed, then it is closed.
\end{remark}

Now, we focus our attention in some known properties about the functional $\mathcal E$.

\begin{defn}
	We say that a map $\mathcal E:\X \longrightarrow \mathbb{R}$ is weakly
	sequentially lower semi-continuous (respectively,  weakly
	sequentially continuous) in $\mathcal A \subset \X$ if for all $x \in\mathcal A$ and
	for all $x_{m}\in \mathcal A$ such that $x_{m}\rightharpoonup v$, it holds
	$\mathcal E(x) \le \lim \inf_{m\rightarrow \infty} \mathcal E(x_m)$ (respectively,
	$\mathcal E(v) = \lim_{m \rightarrow \infty} \mathcal E(x_m)).$
\end{defn}

\begin{prop}{\cite[Proposition 41.8 (H1)]{Zeidler}}\label{Zeid}
	Let $\X$ be a reflexive Banach space and let $\mathcal E:\X \rightarrow \Rbb$
	be a functional. If $\mathcal E$ is a convex and lower semi-continuous functional, then $\mathcal E$ is
	weakly sequentially lower semi-continuous.
\end{prop}

One interesting consequence of Lemma~\ref{lemma_assumption}(b) is the following corollary.

\begin{cor}
Let $\X$ be a reflexive Banach space and let $\mathcal E:\X \rightarrow \Rbb$
	be a functional satisfying (A) and (B). Then $\mathcal E$ is
	weakly sequentially lower semi-continuous and coercive.
\end{cor}

The next theorem says us that in order to have a solution of \eqref{eq:inf_problem}, 
the set $\mathcal A \subset \X$ should be weakly closed.

\begin{theorem}{\cite[Theorem 2]{FN}}\label{th:closed_coercive}
Let $\X$ be a reflexive Banach space, and $\mathcal A \subset \X$ a weakly
	closed set. If $\mathcal E:\mathcal A \rightarrow \Rbb \cup\{\infty\}$ is weakly
	sequentially lower semi-continuous and coercive on $\mathcal A$, then problem
	\eqref{eq:inf_problem} has a solution.
\end{theorem}

\section{Radial dictionaries in reflexive Banach spaces}\label{sec2}

The aim of this section is to introduce the class of sets that we use in 
\eqref{eq:inf_problem} to construct the solution 
of \eqref{main}. We will call these sets radial (universal) dictionaries.

Recall that a dictionary $\Sc \subset \X,$ where $\X$ is a particular Banach space, 
is usually defined as a countable family of unit vectors and it is complete when 
the closure of $\mathrm{span}\, \D$ is equal to $\X.$ 
In some sense, a complete dictionary is a redundant basis 
of a vector space. 

In this paper, we introduce a different definition, 
where essentially the countable family of unit vectors is substituted 
by a closed set of norm one elements. Our definition takes 
into account two properties. The first
one is geometric: a radial dictionary should 
be a cone (it contains all the straight lines 
generated by its owns elements). 
The second property is topological: the whole cone should be
generated by some non-empty closed and bounded set.
As we will see below, any radial radial dictionary is a weakly closed cone, and as consequence, 
the aforementioned set can be identified 
with a cone generated by a non-empty closed
set of norm one elements. Before to introduce the concept of 
radial dictionary we need
the following definitions.

\begin{defn}
A set $C \subset \X$ is an extended cone if for all $x \in C$
the one dimensional subspace $\mathrm{span}\,\{x\} \subset C,$ that is,
$\lambda x \in C$ for all $\lambda \in \mathbb{R}.$
\end{defn}

\begin{defn}
Given a non-empty set $A \subset \X$ we define the extended cone generated by $A,$ 
denoted by $\langle A \rangle,$ as
$$
\langle A \rangle = \{\lambda a: \lambda \in \mathbb{R} \text{ and } a \in A\}.
$$
\end{defn}

Observe that for any  non-empty set $A \subset \X$ it holds  $\langle \langle A \rangle \rangle = \langle A \rangle.$ Moreover,
in order to construct the closed subspace containing the set $A$ we can follow
the next scheme given by the inclusions
$$
A \subset \langle A \rangle \subset \mathrm{span}\,  \langle A \rangle \subset \overline{ \mathrm{span}\, \langle A \rangle}^{\|\cdot\|}.
$$ 
In consequence, the cone generating by the set $A$ is the first geometric object that appears towards the construction
of the closed linear subspace $\overline{ \mathrm{span}\, \langle A \rangle}^{\|\cdot\|}.$

\begin{defn} A subset $\D$ is called a radial dictionary of $\X$ if it is a 
cone generated by some non-empty closed and bounded set in $\X$, that is, there exists a closed
and bounded set $\mathcal{K}\subset \X$ such that $\Sc = \langle \mathcal{K} \rangle.$
\end{defn}

Since in a Banach space any compact set is closed and bounded we have the following consequence.

\begin{lem}\label{compact}
	Let $\X$ be a reflexive Banach space and $\mathcal{K} \subset \X$ be a compact set. 
	Then the set
	$
	\langle \mathcal{K}\rangle
	$
	is a radial dictionary.
\end{lem}

The next result characterizes a radial dictionary as a weakly closed 
cone.

\begin{theorem}\label{cone_char}
Let $\X$ be a reflexive Banach space and  $\Sc \subset \X.$ 
Then the following statements are equivalent.
\begin{enumerate}
\item[(a)] $\Sc$ is a radial dictionary of $\X.$
\item[(b)] $\Sc$ is a weakly closed cone in $\X.$
\end{enumerate}
\end{theorem}

\begin{proof}

Assume $\Sc=\langle \mathcal K \rangle$ is a radial dictionary for some $\mathcal K \subset \X$ closed and bounded. 
We only need to prove to prove that $\Sc$ is weakly closed.  To this end, 
take a sequence $(f_n)_{n \in \mathbb{N}} \subset \Sc$ 
	be such that $f_n \rightharpoonup f.$ Then $f_n=\lambda_n\, h_n$, for 
	$\lambda_n \in \mathbb{R}$ and $h_n \in \mathcal{K},$ is bounded in $\X$ 
	and hence $(\lambda_n)_{n \in \mathbb{N}}$ is bounded
	in $\mathbb{R}$ because $\mathcal{K}$ is. In consequence, the
	sequence $(\lambda_n,h_n)_{n \in \mathbb{N}} \subset \mathbb{R} \times \X$ 
	is also bounded, and there exists a subsequence, also denoted by 
	$(\lambda_n,h_n)_{n \in \mathbb{N}},$ such that
	$\lim_{n\rightarrow \infty} (\lambda_n,h_n) = (\lambda,h).$ Moreover,
	since $\mathbb{R} \times \mathcal{K}$ is closed in $\mathbb{R} \times \X,$
	then $(\lambda,h) \in \mathbb{R} \times \mathcal{K}.$
	
	Now, the bilinear map $\Phi:\mathbb{R} \times \X \longrightarrow \X$ given
	by $\Phi(\lambda,x)=\lambda \, x$ is clearly continuous. Then
	$$\lim_{n \rightarrow \infty} f_n = \lim_{n \rightarrow \infty}\Phi(\lambda_n,h_n) = \Phi(\lim_{n\rightarrow \infty} (\lambda_n,h_n)) = \Phi(\lambda,h) = \lambda \, h,$$
	and hence $f=\lambda\, h,$ because $f_n \rightharpoonup \lambda \, h.$ This proves (b)

	Assume that  $\Sc \subset \X$ is a weakly closed cone and hence $\Sc$ is also closed. 
Consider the closed and bounded set $S_{\X} \subset \X.$ 
Then $\mathcal{K} = \Sc \cap S_{\X}$ is closed
and bounded in $\X.$ Clearly, 
$\langle \mathcal{K} \rangle  \subset \Sc.$
To conclude, take $x \in \Sc,$ if $x = 0$ then $x \in \langle \mathcal{K} \rangle,$ otherwise $\mathrm{span}\,\{x\} \subset \Sc,$
and then $x/\|x\| \in \mathcal{K}.$ Hence $x \in \langle \mathcal{K} \rangle$ and $\langle \mathcal{K} \rangle = \Sc$ and the proof is done.

\end{proof}

From the proof of Theorem~\ref{cone_char} we deduce that every radial dictionary of $\X$
can be written as
$\Sc = \langle \Sc \cap S_{\X} \rangle,$
where $\Sc \cap S_{\X}$ is a closed set of norm one elements.

\begin{defn}\label{dic}
	A subset $\D$ is called universal dictionary of $\X,$ if $\Sc$ is a radial dictionary of $\X$ with the following extra condition:
the linear subspace $\text{span }\mathcal D$ is dense in $\X.$
\end{defn}

Of course, every universal dictionary is indeed a radial dictionary, but the converse 
could be false (see Section~\ref{ex1} below). However, it is easy to see that given 
a radial dictionary $\Sc \subset \X,$ taking into account the Banach 
space $\overline{\mathrm{span}\, \Sc}^{\|\cdot\|_\X},$ generates an universal dictionary 
as the following corollary shows.

\begin{cor}
	Let $(\Y,\|\cdot\|_\Y)$ be a reflexive Banach space and $\mathcal{K} \subset \Y$ be a closed and bounded
	set. Then the set
	$\langle \mathcal{K} \rangle$
	is an universal dictionary of $\X = \overline{\mathrm{span\, \langle \mathcal{K} \rangle}}^{\|\cdot\|_\Y} \subset \Y.$
\end{cor}

Some examples of universal dictionaries are the following.

\subsection{Basis based radial dictionaries}\label{ex1}
Let us consider $(\Y,\|\cdot\|)$ be a reflexive Banach space and $\mathcal{K}=\{x_1,\ldots,x_n\} \subset \Y$ be a set of $n$-linearly independent
	vectors in $\Y.$ Clearly, it is a closed and bounded set. Thus
	$$
	\langle \mathcal{K} \rangle=\{\lambda \, x_i: \lambda \in \mathbb{R} \text{ and } 1 \le i \le n \}
	$$
	is an universal dictionary of  $\X_n = \mathrm{span}\,\langle \mathcal{K} \rangle.$ 

	Assume now that $(\Y,\|\cdot\|)$ is a separable Hilbert
	space. Consider $\mathcal{B}=\{(g_i)_{i \in \mathbb{N}}: \|g_i\|=1\}$ be an orthogonal basis of $\Y.$ Then 
	$\mathcal{K} = \overline{\mathcal{B}}^{\|\cdot\|}$ is closed and bounded in $\Y$ and hence
	$\langle \mathcal{K}\rangle$ is an universal dictionary of $\Y.$

\subsection{A natural radial dictionary for a non-linear convex problem}\label{ex1.5}

Let $\X$ be a reflexive Banach space. Assume 
that $\mathcal E:\X \longrightarrow \mathbb{R}$ satisfies (A)-(B) and $u^*\in \X$ 
solves \eqref{main}.
Consider the closed set
$$
\Omega_0 =\{u \in \X: \mathcal E(u^*) \le \mathcal E(u) \le \mathcal E(0)\} = \mathcal E^{-1}([\mathcal E(u^*),\mathcal E(0)]),
$$
where we can use Proposition \ref{lem:J_decrease}  to define the set above taking the 0. From Lemma~\ref{lemma_assumption}(c), $\Omega_0$ is also bounded. Then $\langle \Omega_0 \rangle$
is an universal dictionary of $\X_{\mathcal E} = \overline{\mathrm{span}\, \langle \Omega_0 \rangle}^{\|\cdot\|}.$
Moreover,
$$
\mathcal E(u^*)=\min_{x \in \X}\mathcal E(x) = \min_{x \in \X_{\mathcal E}}\mathcal E(x).
$$

\subsection{A Radial dictionary for an Approximation Laplacian Fractional Algorithm}\label{ex2}
In \cite{CANDELA1} the authors consider the weakly closed cone
$$
\Sc =\{u \in L^2([0,1]): u(x)=\lambda \, x^{\beta} \text{ for } \lambda \in \mathbb{R} \text{ and } \beta \in [0,2]\}
$$
in the Hilbert space $L^2[(0,1])$ to explain the convergence of a Approximation Laplacian Fractional Algorithm (ALFA). 
This algorithm was previously introduced in the framework of a 
fractional deconvolution model used in image restoration \cite{CANDELA2}. Observe that
$$
\|u\|_{L^2} = \frac{|\lambda|}{\sqrt{1+2\beta}} \text{ holds for all } u = \lambda \, x^{\beta} \in \Sc.
$$
Thus, $\Sc = \langle \mathcal{K} \rangle,$ where
$$
\mathcal{K} = \{u \in L^2([0,1]): u(x)= x^{\beta} \text{ for } \beta \in [0,2]\},
$$
is a closed and bounded set in $L^2[0,1].$ We conclude that $\Sc$ is an universal dictionary for the Hilbert space
$\overline{\mathrm{span}\, \Sc}^{\|\cdot\|_{L^2}}.$

\subsection{Tensor based radial dictionaries}\label{ex3}
Let $D:=\{1,2,\ldots,d\}$ be  a finite index set and $(V_{\alpha},\|\cdot\|_{\alpha})$ be a Banach space
for each index $\alpha \in D.$  we refer to Greub \cite{Greub} for 
the definition of the algebraic tensor space $_{a}\bigotimes
_{\alpha \in D} V_{\alpha}$ generated from vector 
spaces $V_{\alpha}$ $\left(\alpha \in D \right).$
As underlying field we choose $\mathbb{R},$ the suffix `$a$' in 
$_{a}\bigotimes_{\alpha \in D} V_{\alpha}$ 
refers to the `algebraic' nature of the tensor space. By
definition, all elements of 
\begin{equation*}
\mathbf{V}_D:=\left. _{a}\bigotimes_{\alpha \in D} V_{\alpha}\right.
\end{equation*}
are \emph{finite} linear combinations of elementary tensors $\mathbf{v}%
=\bigotimes_{\alpha \in D} v_{\alpha}$ $\left( v_{\alpha}\in V_{\alpha}\right) .$ 

Next, we introduce the injective norm $\left\Vert
\cdot\right\Vert
_{\vee}$ for $\mathbf{v}\in\mathbf{V}=\left.  _{a}%
\bigotimes_{\alpha \in D} V_{\alpha}\right.,$ by%
\begin{equation}
\left\Vert \mathbf{v}\right\Vert _{\vee}:=\sup\left\{
\frac{\left\vert \left(
\varphi^{(1)}\otimes\varphi^{(2)}\otimes\ldots\otimes\varphi^{(d)}\right)
(\mathbf{v})\right\vert }{\prod_{\alpha \in D}\Vert\varphi^{(\alpha)}\Vert_{\alpha}^{\ast}%
}:0\neq\varphi^{(\alpha)}\in V_{\alpha}^{\ast},\, \alpha \in D\right\}  .
\label{(Norm ind*(V1,...,Vd)}%
\end{equation}
Assume that $\|\cdot\|$ is another norm in $\mathbf{V}_D$ not weaker than the injective norm,
that is, there exists a constant $C$ be such that $\|\mathbf{v}\|_{\vee} \le C \|\mathbf{v}\|$
holds for all $\mathbf{v} \in \mathbf{V}_D.$

Let us consider $\Y= \overline{ \mathbf{V}_D}^{\|\cdot\|}$ 
a tensor Banach space given by 
the completion of $\mathbf{V}_D$ under $\|\cdot\|.$ Given $\mathbf{r}=(r_{\alpha})_{\alpha\in D} \in \mathbb{N}^{\#D}$ 
(where $r_{\alpha} \ge 1$ for all $\alpha  \in D$), we introduce the set of tensors in Tucker format with bounded rank $\mathbf{r}$ as
$$
	\mathcal{M}_{\le \mathbf{r}}\left( \mathbf{V}_D\right)=\left\{
	\mathbf{v} \in  \mathbf{V}_D: \begin{array}{c}
	\mathbf{v} \in \left._a\bigotimes_{\alpha\in D} U_{\alpha} \right. \text{ where } U_{\alpha} \text{ is a subspace} \\
	\text{ in } V_{\alpha} \text{ with } \dim U_{\alpha} \le r_{\alpha} \text{ for each } \alpha \in D
	\end{array}
	\right\}.
	$$
In \cite{FalcoHackbusch} (Proposition 4.3) the following result has been proved.
\begin{prop}
\label{tr weakly closed}Let $(\overline{ \mathbf{V}_D}^{\|\cdot\|},\|\cdot\|)$ be a
Banach tensor space with a norm not weaker than the injective norm. 
Then the set $\mathcal{M}_{\le \mathbf{r}}\left(\mathbf{V}_D\right)$ is
weakly closed.
\end{prop}
The set $\mathcal{M}_{\le \mathbf{r}}\left(\mathbf{V}_D\right)$ is clearly a cone and, by the above proposition, it is a weakly closed cone.
Theorem~\ref{cone_char} implies that the set $\mathcal{M}_{\le \mathbf{r}}\left(\mathbf{V}_D\right)$ is a radial dictionary. The linear subspace $\mathrm{span} \,\mathcal{M}_{\le \mathbf{r}}\left(\mathbf{V}_D\right)$ dense in $\Y,$ because it contains the whole set of elementary tensors. Thus, $\mathcal{M}_{\le \mathbf{r}}\left(\mathbf{V}_D\right)$ is an universal dictionary in $\Y.$

\subsection{Neural networks based radial dictionaries}\label{ex4}

To introduce the radial dictionary composed of Neural Networks it 
necessary to distinguish between a \emph{neural network}
as a set of weights and the associated function implemented by the network,
which we call its \emph{realization}. To explain this distinction, 
let us fix numbers $L, N_0, N_1, \dots, N_{L} \in \N$.
We say that a family $\Phi = \big( (A_\ell,b_\ell) \big)_{\ell = 1}^L$ of matrix-vector tuples
of the form $A_\ell \in  \R^{N_{\ell} \times N_{\ell-1}}$ and $b_\ell \in \R^{N_\ell}$
is a \emph{neural network}.

We call $S\coloneqq(N_0, N_1, \dots, N_L)$ the \emph{architecture} of $\Phi$;
furthermore $N(S)\coloneqq \sum_{\ell = 0}^L N_\ell$ is called the \emph{number of neurons of $S$}
and $L = L(S)$ is the \emph{number of layers of $S$} and $d\coloneqq N_0$ the \emph{input dimension} of $\Phi.$ 
For a given architecture $S$, we denote by $\cN(S)$
the \emph{set of neural networks with architecture $S$}.

Defining the realization of such a network $\Phi = \big( (A_\ell,b_\ell) \big)_{\ell=1}^L$
requires two additional ingredients: a so-called \emph{activation function}
$\varrho : \R \to \R$, and a domain of definition $\Omega \subset \R^{N_0}$.
Given these, the \emph{realization of the network} $\Phi = \big( (A_\ell,b_\ell) \big)_{\ell=1}^L$
is the function
\begin{align*}
  \Realization_\varrho^\Omega \left( \Phi \right) :
  \Omega \to \R , \ \
  x \mapsto x_L \, ,
\end{align*}
where $x_L$ results from the following scheme:
\begin{equation*}
  \begin{split}
    x_0 &\coloneqq x, \\
    x_{\ell} &\coloneqq \varrho(A_{\ell} \, x_{\ell-1} + b_\ell),
    \quad \text{ for } \ell = 1, \dots, L-1,\\
    x_L &\coloneqq A_{L} \, x_{L-1} + b_{L},
  \end{split}
\end{equation*}
and where $\varrho$ acts component-wise;
that is, $\varrho(x_1,\dots,x_d) := (\varrho(x_1),\dots,\varrho(x_d))$.

Before to give the next result, let us note that the set $\cN(S)$ of all neural networks
(that is, the network weights) with a fixed architecture forms a finite-dimensional vector space,
which we equip with the norm
\[
  \|\Phi\|_{\cN(S)}
  \coloneqq \| \Phi \|_{\mathrm{scaling}} + \max_{\ell = 1,\dots,L} \|b_\ell\|_{\max}
\]
\[
  \qquad \text{for} \qquad \Phi = \big( (A_\ell, b_\ell) \big)_{\ell=1}^L \in \cN (S) ,
\]
where $\| \Phi \|_{\mathrm{scaling}} \coloneqq \max_{\ell = 1,\dots,L } \| A_\ell \|_{\max}$.
If the specific architecture of $\Phi$ does not matter,
we simply write $\|\Phi\|_{\mathrm{total}}\coloneqq  \|\Phi\|_{\cN(S)}$.
In addition, if $\varrho$ is continuous, we denote the realization map by
\begin{equation}
  \Realization^{\Omega}_{\varrho} :
  \cN(S) \to     C(\Omega ; \R^{N_L}), ~
  \Phi                 \mapsto \Realization^{\Omega}_{\varrho} (\Phi).
  \label{eq:RealizationMapping}
\end{equation}
In \cite{Petersen1} the authors proves the following result (Proposition 3.5).

\begin{prop}\label{prop:bdWeights}
  Let $S = (d, N_1, \dots, N_L)$ be a neural network architecture,
  let $\Omega \subset \R^d$ be compact, let furthermore $p \in (0,\infty)$,
  and let $\varrho: \R \to \R$ be continuous.
  For $C > 0$, let
  \[
    \Theta_C := \big\{ \Phi\in \cN(S):\|\Phi\|_{\mathrm{total}}\leq C \big\}.
  \]
  Then the set $\mathrm{R}_{\varrho}^{\Omega}(\Theta_C)$ is compact in $C(\Omega)$
  as well as in $L^p_{\mu}(\Omega)$, for any finite Borel measure $\mu$ on $\Omega$
  and any $p \in (0, \infty)$.
\end{prop}

Under the assumptions of Proposition~\ref{prop:bdWeights} let 
us consider the compact set $\mathrm{R}_{\varrho}^{\Omega}(\Theta_C) \subset L^p(\mu)$ for $p \in (1,\infty).$
Then, from Lemma~\ref{compact}, the set $\langle \mathrm{R}_{\varrho}^{\Omega}(\Theta_C) \rangle$ is a radial dictionary in 
$L^p(\mu)$ for $p \in (1,\infty).$
Recall that
$$
\langle\mathrm{R}_{\varrho}^{\Omega}(\Theta_C) \rangle = \{\lambda  \Realization^{\Omega}_{\varrho} (\Phi) : \lambda \in \mathbb{R} \text{ and } \Phi \in \Theta_C \}.
$$
Since $\lambda  \Realization^{\Omega}_{\varrho} (\Phi) (x_0) = \lambda x_L = \lambda A_{L} \, x_{L-1} + \lambda  b_{L},$ we obtain
that $\lambda  \Realization^{\Omega}_{\varrho} (\Phi) = \Realization^{\Omega}_{\varrho} (\Phi^*)$ with
$\Phi^* =  \big( (A_\ell^*,b_\ell^*) \big)_{\ell=1}^L$ where $A_\ell^* = A_\ell,$  $b_\ell^* = b_\ell$ for $1 \le \ell \le L-1$
and $A_L^* = \lambda A_L,$  $b_L^* = \lambda b_L.$  In consequence, the radial dictionary $\langle\mathrm{R}_{\varrho}^{\Omega}(\Theta_C) \rangle \subset \Realization^{\Omega}_{\varrho}(\cN(S))$ is a cone included in the set of realizations of neural networks with architecture $S.$
Thus, $\langle \mathrm{R}_{\varrho}^{\Omega}(\Theta_C) \rangle$ is a universal dictionary of the reflexive Banach space $\overline{\mathrm{span}\, \langle\mathrm{R}_{\varrho}^{\Omega}(\Theta_C) \rangle}^{\|\cdot\|_{L^p_{\mu}}} \subset L^p_{\mu}(\Omega)$
for $p \in (1,\infty).$

\subsection{$\mathcal E$-dictionary optimization over $\Sc$}
Now, to complete (S2) we need to introduce the following definition. 

\begin{defn}[Dictionary optimization]
Let $\X$ be a reflexive Banach space and $\Sc$ be radial dictionary of $\X.$
Assume that $\mathcal E:\X \longrightarrow \mathbb{R}$ satisfies (A)-(B). A  $\mathcal E$-dictionary optimization over $\Sc$ is 
a multivalued map defined as follows: $$\boldsymbol{\nabla}_{\cdot}(\mathcal E;\Sc):\X \rightrightarrows \Sc, \quad u \mapsto \boldsymbol{\nabla}_u(\mathcal E;\Sc):= \arg \min_{z  \in u+\Sc}\mathcal E(z).$$ 
\end{defn}

\begin{remark}
	Observe that for radial dictionaries in reflexive Banach spaces, 
$$\boldsymbol{\nabla}_u(\mathcal E;\Sc):= \arg \min_{v \in \Sc}\mathcal E(u+v).$$
	
\end{remark}

The next result gives us some interesting properties of this multivalued map.

\begin{theorem}\label{lem:carac_sol_S1}
Let $\X$ be a reflexive Banach space and $\Sc$ be radial dictionary of $\X.$
Assume that $\mathcal E:\X \longrightarrow \mathbb{R}$ satisfies (A)-(B). Then, for all $u \in \X$, the following statements hold.
\begin{enumerate}
\item[\textbf{(a)}] The set $\boldsymbol{\nabla}_u(\mathcal E;\Sc)\neq \emptyset$ and it is weakly closed in $\X.$
\item[\textbf{(b)}] If $\Sc$ is an universal dictionary in $\X$ and $u^* \in \X$ satisfies $0 \in \boldsymbol{\nabla}_{u^*}(\mathcal E;\Sc)$, then $u^*$ solves \eqref{main}.
\end{enumerate}
\end{theorem}

\begin{proof}

\textbf{(a)} We claim that $u + \Sc$ is also weakly closed.
To see this assume that $u + y_n \rightharpoonup y$ for some $\{y_n\}_{n \ge 1} \subset \Sc,$ then
$y_n \rightharpoonup y-u$ and since
$\Sc$ is weakly closed, $y-u \in \Sc.$ In
consequence $y \in u+ \mathcal D$ and the claim follows.

Next, we will show that $\boldsymbol{\nabla}_u(\mathcal E;\Sc)\neq \emptyset$
for all $u \in \X.$ Applying the Fr\'echet differentiability of $\mathcal E$, we know that $\mathcal E$ is continuous. Since $\mathcal E$ is 
also convex, $\mathcal E$ is weakly sequentially lower semi-continuous by Proposition~\ref{Zeid}. Moreover, $\mathcal E$ is
coercive on $\X$ by Lemma~\ref{lemma_assumption}(c). By	the claim $u + \Sc$ is a weakly
closed subset in $\X.$ Then,  the existence of a minimizer follows from Theorem~\ref{th:closed_coercive},
that is, $\boldsymbol{\nabla}_u(\mathcal E;\Sc)\neq \emptyset.$

Now, we show that $\boldsymbol{\nabla}_u(\mathcal E;\Sc)$ is weakly closed in $\X$. For that, consider 
the map $$\mathcal{E}_u:\X \longrightarrow \mathbb{R}$$ given by
$\mathcal{E}_u(x) = \mathcal E(x+u).$ Let $a = \min_{\xi \in \Sc} \mathcal{E}_u(\xi) \in \mathbb{R}$. Then, $\mathcal{E}_u^{-1}(\{a\})$ is a closed set in $\X,$ thus
$\boldsymbol{\nabla}_u(\Sc;\mathcal{E}) = \mathcal{E}_u^{-1}(\{a\}) \cap \Sc$ is also closed in $\X.$
Consider a sequence $\{u_n\} \subset \boldsymbol{\nabla}_u(\Sc;\mathcal{E})$ such that
$u_n \rightharpoonup u$. Then $\{u_n\}$ is a bounded sequence in $\X$ that is a reflexive Banach space. Hence,  there exists a subsequence, also denoted by $\{u_n\}$, convergent to
	$z \in \boldsymbol{\nabla}_u(\Sc;\mathcal{E}).$ By
	the unicity of limits, $z=u.$ Thus, the set $\boldsymbol{\nabla}_u(\mathcal E;\Sc)$ 
	is weakly closed in $\X$. This ends the proof of (a).

\textbf{(b)}  Finally, to prove the last part of the theorem, assume
that $u^* \in \X$ satisfies $0 \in \boldsymbol{\nabla}_{u^*}(\mathcal E;\Sc).$ Then
for all $\gamma\in\Rbb_+$ and $z\in\Sc$, it holds
$\mathcal E(u^* + \gamma z) \ge \mathcal E(u^*)$ and therefore
$$
\langle \mathcal E'(u^*),z\rangle  = \lim_{\gamma\searrow
	0} \frac{1}{\gamma} ( \mathcal E(u^* + \gamma z) -
 \mathcal E(u^*))\ge 0
$$
holds for all $z\in\Sc.$ Since $\Sc$ is a cone, we have $\langle \mathcal E'(u^*),z\rangle  \; = \; 0$ for all
$z\in\Sc.$ Finally, by using the universality of $\Sc,$ we obtain
$
\langle \mathcal E'(u^*),v\rangle  = 0,$ for all $v\in \X,$
and (b) follows from Theorem \ref{th:solution}.  
\end{proof}


\section{Greedy Sequences by dictionary optimization}\label{sec3}

Let $\X$ be a reflexive Banach space and $\Sc$ be a universal dictionary of $\X.$ 
Assume that $\mathcal E:\X \longrightarrow \mathbb{R}$ satisfies (A)-(B) and $u^*\in \X$ 
solves \eqref{main}. Given $u \in \X,$ from the above theorem, 
if $0\in \boldsymbol{\nabla}_{u}(\mathcal E;\Sc)$
then $u=u^*,$ roughly speaking $\mathcal E'(u^*) = 0.$ Otherwise, $\mathcal E(v) < \mathcal E(u)$ holds for 
all $v \in u+\boldsymbol{\nabla}_u(\mathcal E;\Sc).$ Intuitively, the set
$\boldsymbol{\nabla}_{u}(\mathcal E;\Sc)$ contains feasible optimal corrections to 
reduce the value of $\mathcal E(u+\cdot)$ over the dictionary $\Sc.$
This discussion allows us to propose 
an iterative multivalued procedure that motivates the following definition.

\begin{defn}[Greedy Sequence by Dictionary Optimization]
	\label{def:updated_topological_pgd}
	We say that a greedy sequence of $u^*$ 
	by a $\mathcal E$-dictionary optimization over $\Sc$ is any sequence $\{u_m\}_{m \ge 1} \subset \X$ 
	constructed  by the following greedy algorithm:
\begin{enumerate}
\item $u_0 = 0$; 
\item for $m\geq 1$, $u_m  \in u_{m-1} + \boldsymbol{\nabla}_{u_{m-1}}(\mathcal E;\Sc).$
\end{enumerate}
\end{defn}

Observe that in a practical implementation, $u_m = u_{m-1}+\widetilde{\varepsilon}_m$ for 
some $\widetilde{\varepsilon}_m  \in \boldsymbol{\nabla}_{u_{m-1}}(\mathcal E;\Sc),$ where
\begin{equation} 
d_m:= \mathcal E(u_{m-1}) - \mathcal{E}(u_{m-1}+\widetilde{\varepsilon}_m) \ge 0
\end{equation}
is an optimal constant value over $\Sc.$ 
The quantity $\widetilde{\varepsilon}_m =  u_m - u_{m-1} \in\boldsymbol{\nabla}_{u_{m-1}}(\mathcal E;\Sc)$ 
can be see as a conditional residual that belongs to $\Sc,$ and depends on $ u_{m-1}.$
By definition, the first residual $\widetilde{\varepsilon}_1  = u_1 \in \boldsymbol{\nabla}_{0}(\mathcal E;\Sc)$ is a minimum of $\mathcal E$ over $\Sc,$ where in absence of any a priori information we may assume that $u_{0} =0.$
Moreover, for all $m \ge 1,$ we have $u_m = u_{m-1} + \widetilde{\varepsilon}_m = \sum_{\ell =1}^m  \widetilde{\varepsilon}_{\ell}$ where $\widetilde{\varepsilon}_{\ell} \in \boldsymbol{\nabla}_{u_{\ell-1}}(\mathcal E;\Sc)$ for $\ell \ge 1$ (see Figure~\ref{fig1}). In consequence, the sequence $\{u_m\}_{m \ge 1}$  
can be interpreted as an algebraic stack that progressively stores conditional residues.
\begin{figure}[h]
	\includegraphics[scale=0.18]{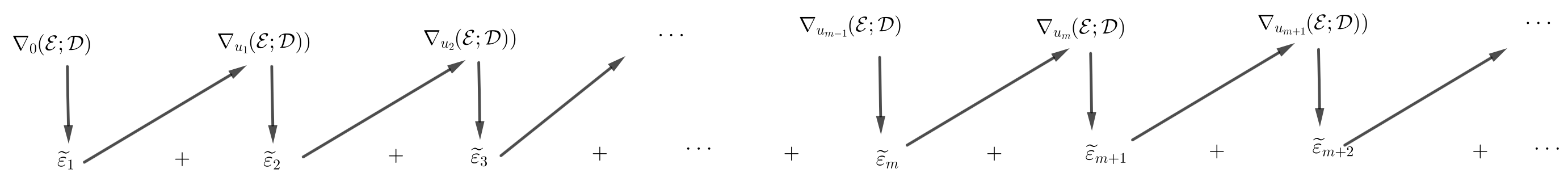}
	\caption{A scheme of the construction of a greedy sequence of $u^*$ by a $\mathcal E$-dictionary optimization over $\Sc.$}\label{fig1}
\end{figure}

A greedy sequence by a dictionary optimization can be see as an 
extension of the Method of Steepest Descend
(see \cite[Chapter XV]{Kantorovich}) because both methods share
assumptions (A)-(B) to assure its convergence. To compare, recall
that Steepest Descend is performed by means a sequence $\{u_m\}_{m \ge 1}$ constructed 
as follows:
\begin{enumerate}
	\item $u_0 = 0,$ and for $m \ge 1,$
	\item given $u_{m-1}$ compute:
	\begin{enumerate}
		\item[(2.1)] $w_m \in \arg \min_{z \in S_{\X}} \langle \mathcal E'(u_{m-1}),z\rangle;$
		\item[(2.2)] $\lambda_m \in \arg \min_{\lambda \in \R} \mathcal E(u_{m-1}+\lambda \, w_m);$
	\end{enumerate}
	put $u_m = u_{m-1}+\lambda_m \, w_m.$
\end{enumerate}
Step (2.1) appears as the key point of this procedure. It  implies $\langle \mathcal E(u_{m-1}),w_m\rangle = -\|\mathcal E'(u_{m-1})\|_*.$ Moreover, a bound of $O(m^{-1})$ for the rate of convergence of the sequence $\{\mathcal E(u_m) - \mathcal E(u^*)\}_{m \ge 1},$
where $u^*$ solves \eqref{main}, has been obtained (see Theorem 3 p. 466 in \cite{Kantorovich}).

We are going to prove two results:
\begin{itemize}
	\item The first one is to show that each greedy sequence $\{u_m\}_{m \ge 1}$ of $u^*$ 
	by a $\mathcal E$-dictionary optimization over $\Sc$ converges in $\X$ to $u^*,$ 
	where $u^*$ is the solution of \eqref{main} as is stated in the next result.
\end{itemize}
	\begin{theorem}\label{th:convergence}
		Let $\X$ be a reflexive Banach space and $\Sc$ be a universal dictionary of $\X.$
		Assume that $\mathcal E:\X \longrightarrow \mathbb{R}$
		satisfies (A)-(B) and $u^*\in \X$ solves \eqref{main}. Then for each greedy sequence $\{u_m\}_{m \ge 1}$ of $u^*$ by a $\mathcal E$-dictionary optimization over $\Sc$ 
		converges in $\X$ to $u^*,$ that is,
		$\lim_{m\rightarrow\infty}\Vert u^* -u_m\Vert
		= 0.$
	\end{theorem}
\begin{itemize}
	\item The second one shows that a greedy sequence by dictionary optimization
	has the same rate of convergence as the well-known Method of the Steepest
	Descend implemented in a reflexive Banach space.
\end{itemize}
\begin{theorem}\label{th:rate_convergence}
		Let $\X$ be a reflexive Banach space and $\Sc$ be a universal dictionary of $\X.$
		Assume that $\mathcal E:\X \longrightarrow \mathbb{R}$
		satisfies (A)-(B) and $u^*\in \X$ solves \eqref{main}. Then each greedy sequence $\{u_m\}_{m \ge 1}$ of $u^*$ by a $\mathcal E$-dictionary optimization over $\Sc$ 
		satisfies
		$$
		\mathcal E(u_m)-\mathcal E(u^*) = O(m^{-1}).
		$$
	\end{theorem}

To prove Theorem \ref{th:convergence}, the next proposition is needed. It gives us some useful
 properties of the sequences
$\{u_m\}_{m \in \N},$  $\{\mathcal E({u}_m)\}_{m\ge
		1}$  and  $\{\mathcal E'({u}_m)\}_{m\ge
		1}.$

\begin{prop}\label{lem:J_decrease}
	Let $\X$ be a reflexive Banach space and $\Sc$ be a universal dictionary of $\X.$
	Assume that $\mathcal E:\X \longrightarrow \mathbb{R}$
	satisfies (A)-(B) and $u^*\in \X$ solves \eqref{main}. Then for each greedy sequence $\{u_m\}_{m \ge 1}$ of $u^*$ by a $\mathcal E$-dictionary optimization over $\Sc$ the
	following statements hold.
	\begin{enumerate}
	\item[\textbf{(a)}] The sequence $\{\mathcal E({u}_m)\}_{m\ge
		1},$ is non increasing and bounded below, that is, $$\mathcal E(u^*) \le \mathcal E({u}_{m})\le \mathcal E({u}_{m-1}),$$ 
		holds for all $m \ge 1.$
	\item[\textbf{(b)}] If $\lim_{m\to\infty}\mathcal E(u_m) = \mathcal E(u^*),$ then $\lim_{m\to\infty}\Vert u^* -u_m\Vert = 0.$
	\item[\textbf{(c)}] If for some $m \ge 1$  it holds $u_{m} = u_{m-1}$ then $u_{m-1} = u^*.$ 
\item[\textbf{(d)}] The sequence
$\{\mathcal E'(u_m)\}_{m\in\mathbb{N}}$ satisfies $
\lim_{m\rightarrow \infty}\langle \mathcal E'(u_{m}), z
\rangle = 0, $ for all $z$ in $\X.$
\item[\textbf{(e)}] There exists a 
subsequence $\{u_{m_k}\}_{k\in \mathbb{N}}$
such that
$\lim_{k \rightarrow \infty}\langle \mathcal E'(u_{m_k}),u_{m_k} \rangle \rightarrow 0.$
\end{enumerate}
\end{prop}
\begin{proof}
See Appendix~\ref{appendix:a1}.
\end{proof}

Now, we give the proof of the theorem.

\subsubsection*{Proof of Theorem~\ref{th:convergence}}
From Proposition~\ref{lem:J_decrease}(a), $\{\mathcal E(u_{m})\}$ is a non
increasing sequence. If there exists $m$ such that
$\mathcal E(u_{m})=\mathcal E(u_{m-1})$. From Proposition~\ref{lem:J_decrease}(c), 
we have $u_m=u^*$, which ends
the proof. Let us now suppose that
$\mathcal E(u_m)<\mathcal E(u_{m-1})$ for all $m$. Now, $\{\mathcal E(u_m)\}_{m \in \N}$
is a strictly decreasing sequence of real numbers which is bounded below by
$\mathcal E(u^*)$. Then, there exists
$$
\mathcal E^*=\lim_{m\to \infty} \mathcal E(u_m) \ge \mathcal E(u^*).
$$
If $\mathcal E^*=\mathcal E(u^*)$, Proposition~\ref{lem:J_decrease}(b) allows
to conclude that $\{u_m\}$ strongly converges to
$u^*$. Therefore, it remains to prove that
$\mathcal E^*=\mathcal E(u^*)$. As a consequence of \eqref{eq:J_strong_convexity_consequence1} we have
\begin{equation}\label{eq_theo}
\mathcal E(u_{m})-\mathcal E(u^*) \le \langle
\mathcal E'(u_{m}),u_{m}-u^*\rangle = \langle
\mathcal E'(u_{m}),u_{m}\rangle -    \langle
\mathcal E'(u_{m}), u^*\rangle
\end{equation}
By Proposition~\ref{lem:J_decrease}(e), we have that there exists a
subsequence $\{u_{m_k}\}_{k\in\mathbb{N}}$ such that $$\lim_{k \rightarrow \infty}\langle
\mathcal E'(u_{m_k}),u_{m_k}\rangle = 0$$ and, from Proposition~\ref{lem:J_decrease}(d), also
$\lim_{k \rightarrow \infty} \langle \mathcal E'(u_{m_k}), u^*\rangle = 0.$
Therefore, putting $\{u_{m_k}\}_{k\in\mathbb{N}}$ in \eqref{eq_theo} and
taking limits as $k\rightarrow \infty,$ we obtain
$$
\mathcal E^* - \mathcal E(u^*) = \lim_{k\to\infty } \mathcal E(u_{m_k}) -
\mathcal E(u^*) \le 0
$$
Since we already had $\mathcal E^*\ge \mathcal E(u^*)$, this yields
$\mathcal E^*=\mathcal E(u^*)$, which ends the proof. 
\qed

To prove now Theorem \ref{th:rate_convergence}, we need the following three results. The first
two lemmas are due to Akilov and  Kantorovich \cite{Kantorovich}.

\begin{lem}[{\cite[Lemma 2 p. 464]{Kantorovich}}]\label{rate1}
Let $\X$ be a reflexive Banach space and assume that $\mathcal E:\X \longrightarrow \mathbb{R}$
satisfies (A). Assume that $\mathcal{E}'$ satisfies the Lipschitz condition in 
$B_{\X,r+s}$ for some $r,s > 0$ with a Lipschitz constant equal to $\mathcal L.$ If $z \in B_{\X,s}$ then it holds
$$
\mathcal E(x+z) \le \mathcal E(z)+\langle\mathcal E'(x),z \rangle + \frac{\mathcal{L}}{2}\|z\|^2.
$$ 
\end{lem}

\begin{lem}[{\cite[Lemma 4 p. 467]{Kantorovich}}]\label{rate2}
Let $\{\lambda_m\}_{m \in \N}$ be a sequence of strictly positive real numbers such that,
for some $\mu > 0$ we have $\lambda_m -\lambda_{m+1} \ge \mu \lambda_m^2$ for all
$m \ge 1.$ Then $\lambda_m = O(m^{-1}).$
\end{lem}

The next proposition provides an inequality that will be useful to prove Theorem~\ref{th:rate_convergence}.

\begin{prop}\label{rate3}
Let $\X$ be a reflexive Banach space. Assume that $\mathcal E:\X \longrightarrow \mathbb{R}$
satisfies (A)-(B) and $u^*\in \X$ solves \eqref{main}. Let
$\Omega_0 = \mathcal{E}^{-1}([\mathcal E(u^*),\mathcal E(0)]),$ then 
there exists $c = c(\Omega_0) > 0$ such that 
$$
\mathcal E(u) - \mathcal E(u^*) \le c \|\mathcal E'(u)\|_*
$$
holds for all $u \in \Omega_0.$
\end{prop}

\begin{proof}
See Appendix~\ref{appendix:a1}.
\end{proof}

From the above result and Proposition~\ref{lem:J_decrease}(a) we obtain the following.

\begin{cor}\label{rate4}
Let $\X$ be a reflexive Banach space and $\Sc$ be a universal dictionary of $\X.$
Assume that $\mathcal E:\X \longrightarrow \mathbb{R}$
satisfies (A)-(B) and $u^*\in \X$ solves \eqref{main}. Let
$\Omega_0 = \mathcal{E}^{-1}([\mathcal E(u^*),\mathcal E(0)]),$ then there exists $c = c(\Omega_0) > 0$ 
such that for each greedy sequence $\{u_m\}_{m \ge 1}$ of $u^*$ by a dictionary optimization over $\Sc,$ 
$$
\mathcal E(u_m) - \mathcal E(u^*) \le c \|\mathcal E'(u_m)\|_*
$$
holds for $m \ge 1.$
\end{cor}

Now, we have all ingredients to prove Theorem~\ref{th:rate_convergence}

\subsubsection*{Proof of Theorem~\ref{th:rate_convergence}}
Put $\lambda_m = \mathcal E(u_{m})-\mathcal E(u^*)$ for $m \ge 1.$  We claim that the sequence
$(\lambda_m)_{m \ge 1}$ satisfies the hypothesis of Lemma~\ref{rate2}. Assume
that $u_{m+1} = u_{m} + \lambda_{m+1} w_{m+1}$ where $\|w_{m+1}\| = 1$ 
and $\lambda_{m+1}\in \mathbb{R}.$ 

Recall that $\Omega_0$ is closed and bounded (see Section~\ref{ex1.5}). Then the set
$$
\Omega_0 -\Omega_0 =\{z \in X: z=z_1-z_2 \text{ where } z_i \in \Omega_0 \text{ for } i=1,2\}
$$
is also bounded. Fix $s > 0$ be such that
$\Omega_0 -\Omega_0 \subset B_{\X,s}.$ Take $a=r+s$ for some $r >0.$  Let $\mathcal L$ be the Lipschitz
continuity constant of ${\mathcal E}'$ on the bounded set $B_{\mathbb X,a}.$
Then, for $m \ge 1,$
$$
\|\mathcal{E}'(u_m)\|_* = \|\mathcal{E}'(u_m)-\mathcal{E}'(u^*)\| \le  \mathcal L \|u_m -u^*\| \le \mathcal Ls,
$$
holds.

Now, from Lemma~\ref{rate1}, 
$$
\begin{array}{cl}
\mathcal E(u_{m+1}) & \le \mathcal E(u_{m}+\mu w_{m+1}) \\ 
&  \le \mathcal E(u_m) + \mu \langle \mathcal{E}'(u_m),w_{m+1} \rangle
+ \frac{\mathcal L}{2} \mu^2 \\
& \le \mathcal E(u_m) + \mu \|\mathcal{E}'(u_m)\|_* + \frac{\mathcal L}{2} \mu^2 
\end{array}
$$
holds for $\|\mu w_{m+1}\| = |\mu| \le s,$ in particular $|\mu_{m+1}| = \|u_{m+1}-u_{m}\| \le s.$
The minimum for
$$
\gamma(\mu) = \mathcal E(u_m) + \mu \|\mathcal{E}'(u_m)\|_* + \frac{\mathcal L}{2} \mu^2
$$
is obtained at $\widetilde{\mu} = -\frac{1}{\mathcal L} \|\mathcal E'(u_m)\|_*,$ where $|\widetilde{\mu}| \le s,$ and
$\gamma(\widetilde{\mu}) =  \mathcal E(u_m) - \frac{1}{2\mathcal L}\|\mathcal E'(u_m)\|_*^2.$ Hence,
$$
\mathcal E(u_{m+1}) \le  \mathcal E\left(u_{m}  -\frac{1}{\mathcal L} \|\mathcal E'(u_m)\|_* w_{m+1}\right) \le  \mathcal E(u_m) - \frac{1}{2\mathcal L}\|\mathcal E'(u_m)\|_*^2.
$$
Thus,
$$
 \lambda_m - \lambda_{m+1} = \mathcal E(u_{m})  - \mathcal E(u_{m+1}) \ge \frac{1}{2\mathcal L} \|\mathcal E'(u_m)\|_*^2.
$$
On the other hand, by Corollary~\ref{rate4}, we have for some $c > 0$
$$
\lambda_m = \mathcal E(u_m) - \mathcal E(u^*) \le c \|\mathcal E'(u_m)\|_*,
$$
for $m \ge 1.$ We conclude that 
$$
 \lambda_m - \lambda_{m+1} \ge \frac{1}{2c^2 \mathcal L} \lambda_m^2.
$$
Then the theorem follows from Lemma~\ref{rate2} taking $\mu =  \frac{1}{2c^2 \mathcal L}.$
\qed

\section{Conclusions and final remarks}\label{sec4}

In this paper we propose a general mathematical framework to solve non-linear convex
problems by a greedy sequence over a dictionary optimization. It appears as
a supervised learning approach  to the classical Method of Steepest Descend in
a reflexive Banach space. 
Moreover, we introduce the radial dictionaries as a 
class of sets generated by closed and bounded sets. This class includes, among others, 
tensors in Tucker format with bounded rank and 
Neural Networks with fixed architecture and bounded parameters.
We also point out that in our framework the convergence of the
Proper Generalized Decomposition (PGD) in a reflexive Banach space, 
that has been proved in \cite{FN}, is fulfilled. The novelty is that 
we are able to provide a convergence rate of $O(m^{-1})$ for the PGD.

\appendix

\section{Proof of Proposition~\ref{lem:J_decrease} and Proposition~\ref{rate3}}\label{appendix:a1}

To prove Proposition~\ref{lem:J_decrease} we need the following technical lemma.

\begin{lem}\label{NEW}
Let $\X$ be a reflexive Banach space and $\Sc$ be a universal dictionary of $\X.$
	Assume that $\mathcal E:\X \longrightarrow \mathbb{R}$
	satisfies (A)-(B) and $u^*\in \X$ solves \eqref{main}. Then for each greedy sequence $\{u_m\}_{m \ge 1}$ of $u^*$ by a $\mathcal E$-dictionary optimization over $\Sc$ the
	following statements hold.
\begin{enumerate}
\item[\textbf{(a)}] The sequence $\{\mathcal E'({u}_m)\}_{m\ge
		1},$ satisfies
	$\langle \mathcal E'(u_{m}), u_m-u_{m-1} \rangle = 0$ for all $m \ge 1.$
	\item[\textbf{(b)}] It holds
	$\sum_{m=1}^\infty \Vert u_m-u_{m-1}  \Vert^s < \infty,$
for some  $s > 1,$ and thus,
\begin{align}\lim_{m\to\infty
} \Vert u_m-u_{m-1} \Vert = 0. \label{eq:lim_zm}
\end{align}
\item[\textbf{(c)}] There exists $C >0$ such
that for $m\ge 1,$
$$ \vert \langle {\mathcal E}'(u_{m-1}),z\rangle \vert \leqslant C \Vert  u_m-u_{m-1}
\Vert^2 \Vert z\Vert,
$$
holds for all $z \in \Sc.$ 
\end{enumerate}
\end{lem}
\begin{proof}
\textbf{(a)} Take $ u_m-u_{m-1} = \widetilde{\varepsilon}_m \in \boldsymbol{\nabla}_{u_{m-1}}(\mathcal E;\Sc)$ for some $m\ge 1.$
Put $\widetilde{\varepsilon}_m = \lambda_m w_m$, with
$\lambda_m\in\Rbb^+$ and $\norm[w_m]=1$. Recall that $\widetilde{\varepsilon}_m 
\in \arg\min_{z\in \Sc} \mathcal E(u_{m-1}+z)$.
Since $\mathcal D$ is a cone, we obtain
$$\mathcal E(u_{m-1}+\lambda_m w_m)\le
\mathcal E(u_{m-1}+\lambda w_m)$$ for all $\lambda \in
\Rbb$. Taking $\lambda = \lambda_m \pm \gamma$,
with $\gamma\in \Rbb^+$, we obtain for all cases
\begin{align*}
0&\le \frac{1}{\gamma} \left(\mathcal E(u_{m-1}+\lambda_m
w_m\pm \gamma w_m)- \mathcal E(u_{m-1}+\lambda_m
w_m)\right).
\end{align*}
Taking the limit $\gamma \searrow 0$, we obtain $ 0 \le \pm \langle
\mathcal E'(u_{m-1}+\lambda_m w_m),w_m\rangle $ and
therefore
$$\langle \mathcal E'(u_{m-1}+\lambda_m w_m),w_m\rangle = 0,$$
which ends the proof.

\textbf{(b)} By the condition (B) of $\mathcal E$, we have
\begin{align}
{\mathcal E}(u_{m-1})-{\mathcal E}(u_{m}) &\ge \langle -
{\mathcal E}'(u_m),u_m-u_{m-1}\rangle +
\frac{\alpha}{s} \Vert u_m -u_{m-1}\Vert^s \nonumber
\end{align}
for some $s > 1$ and $\alpha > 0.$
Using (a), we obtain
\begin{align}
{\mathcal E}(u_{m-1})-{\mathcal E}(u_{m}) &\ge \frac{\alpha}{s} \Vert
 u_m-u_{m-1} \Vert^s \nonumber 
\end{align}
Now, summing on $m$, and using $\lim_{m\to\infty } {\mathcal E}(u_m) =
L <\infty $ (by (a)), we obtain
\begin{align*}
\frac{\alpha}{s} \sum_{m=1}^\infty  \Vert  u_m-u_{m-1} \Vert^s &\le
\sum_{m=1}^\infty ({\mathcal E}(u_{m-1})-{\mathcal E}(u_{m}))  = {\mathcal E}(0) -
L< +\infty.
\end{align*}
which implies $\lim_{m \rightarrow \infty}\Vert  u_m-u_{m-1}
\Vert^s =0.$ The continuity of the map $x \mapsto x^{1/s}$ at $x=0$
proves (\ref{eq:lim_zm}).

\textbf{(c)} By using that $\{\mathcal E(u_m)\}_{m \in \N}$  
is bounded sequence and the functional $\mathcal E$ 
is coercive we obtain that  $\{u_m\}_{m\ge 1}$ is bounded. 
Then there exists $r > 0$ be such that $\|u_m\| \le r$
holds for all $m \ge 1.$ From (b), 
$\Vert u_m-u_{m-1}\Vert \rightarrow 0$ as $m\to\infty,$ hence $\{u_m -u_{m-1}\}_{m\ge 1}$ is also a
bounded sequence and hence there exists $s > 0$ 
be such that $\|u_m-u_{m-1}\| \le s$
holds for all $m \ge 1.$
Take $ a = r+s > 0$ and let $\mathcal L$ be the Lipschitz
continuity constant of ${\mathcal E}'$ on the bounded set $B_{\mathbb X,a}.$ Then
\begin{align*}
-\langle {\mathcal E}'(u_{m-1}),z\rangle  &= \langle
{\mathcal E}'(u_{m-1}+z)- {\mathcal E}'(u_{m-1})
 ,z\rangle - \langle {\mathcal E}'(u_{m-1}+z),z\rangle
\\
& \le \|{\mathcal E}'(u_{m-1}+z)- {\mathcal E}'(u_{m-1})\|_*\|z\| - \langle
{\mathcal E}'(u_{m-1}+z),z\rangle \\ 
& \le \mathcal L \Vert z \Vert^2 - \langle
{\mathcal E}'(u_{m-1}+z),z\rangle
\end{align*}
for all $z\in  B_{\Sc,s} \subset B_{\mathbb X,a}.$ By \eqref{eq:J_strong_convexity_consequence0} and since ${\mathcal E}(u_{m})\le
{\mathcal E}(u_{m-1}+z)$ holds for all $z\in\Sc$, we
have
\begin{align*}
\langle {\mathcal E}'(u_{m-1}+z),(u_m -u_{m-1})-z
\rangle \le {\mathcal E}(u_{m}) -
{\mathcal E}(u_{m-1}+z) \le 0.
\end{align*}
Therefore, for all $z\in  B_{\Sc,s}$, we have
\begin{align*}
-\langle {\mathcal E}'(u_{m-1}),z\rangle  &\le
 \mathcal L
\Vert z \Vert^2 - \langle
{\mathcal E}'(u_{m-1}+z),u_m -u_{m-1}\rangle\\
&\le  \mathcal L \Vert z \Vert^2 - \langle
{\mathcal E}'(u_{m-1}+z) -
{\mathcal E}'(u_{m-1}),u_m -u_{m-1}\rangle
\quad \text{\footnotesize (by (a))} \\
&\le  \mathcal L \Vert z \Vert^2 + \mathcal L\Vert z-(u_m -u_{m-1})
\Vert \norm[u_m -u_{m-1}] \quad \text{\footnotesize (by (b))}
\\
&\le  \mathcal L \left(\norm[z] ^2 +  \norm[z]
\norm[u_m -u_{m-1}] + \norm[u_m -u_{m-1}]^2\right)
\end{align*}
Let $z=w \norm[u_m -u_{m-1}] \in B_{\Sc,s}$, with
$\norm[w] = 1$. Then
\begin{align*}
\vert \langle {\mathcal E}'(u_{m-1}),w\rangle \vert  & \le
3 \mathcal L \norm[u_m -u_{m-1}]^2 \quad \text{ for all } w\in \Sc \cap S_{\X}.
\end{align*}
Taking $w=z/\norm[z]$, with
$z\in \Sc$, and $C=3 \mathcal L > 0$ we obtain
 \begin{align*}
\vert \langle {\mathcal E}'(u_{m-1}),z\rangle \vert  &\le C
\norm[u_m -u_{m-1}]^2\norm[z] \quad \text{ for all }
z\in\Sc.
\end{align*}
This concludes the proof of (c) and the lemma. 
\end{proof}

\subsubsection*{Proof of Proposition~\ref{lem:J_decrease}} \textbf{(a)} Since $u_m = u_{m-1} + \boldsymbol{\nabla}_{u_{m-1}}(\mathcal E;\Sc)$ and $0\in\Sc,$ we have
	$$\mathcal E(u_m)=\min_{\xi \in\Sc}
	\mathcal E(u_{m-1}+\xi) \leq \mathcal E(u_{m-1}).$$
	Then, $\lbrace \mathcal E(u_m)\rbrace_{m\geq 1}$ is non increasing.

\textbf{(b)} By using the ellipticity property
	\eqref{eq:J_strong_convexity_consequence} of $\mathcal E$, we have
	\begin{align}
		\mathcal E(u_m)-\mathcal E(u^*)  \ge \langle \mathcal E'(u^*),u_m-u^*\rangle + \frac{\alpha}{s}\Vert
		u^*-u_{m}\Vert^s = \frac{\alpha}{s}\Vert u^*-u_{m}\Vert^s. \nonumber
	\end{align}
	Therefore,
	$$
	\frac{\alpha}{s} \Vert u^*-u_{m}\Vert^s
	\le \mathcal E(u_m)-\mathcal E(u^*) 
	\underset{m\to \infty}{\longrightarrow} 0,
	$$
	which ends the proof. 	
	
\textbf{(c)} Assume $u_m = u_{m-1}$ for some $m \ge 1.$ Then
	 $u_m -u_{m-1}= 0 \in \boldsymbol{\nabla}_{u_{m-1}}(\mathcal E;\Sc)$. Hence, from Theorem \ref{lem:carac_sol_S1}, we conclude that ${u}_{m-1}$
	solves \eqref{main}.
	
\textbf{(d)} Thanks to the reflexivity of the space $\X$, we can identify
$\X^{**}$ with $\X$ and the
duality pairing $\langle \cdot, \cdot
\rangle_{\X^{**},\X^*}$ with
$\langle \cdot, \cdot
\rangle_{\X^*,\X}$
(recall that weak and weak-$*$ topologies coincide on
$\X^*$).
The sequence $\{u_m\}_{m\in \mathbb{N}}$ being bounded, and since by condition (A)
$\mathcal E'$ is locally Lipschitz continuous, we have that there
exists a constant $C>0$ such that
$$
\Vert \mathcal E'(u_m)\Vert = \Vert \mathcal E'(u_m)
-\mathcal E'(u)\Vert \le C \Vert u-u_m\Vert.
$$
This inequality proves that $\{\mathcal E'(u_m)\}_{m\in\mathbb N}\subset
\X^*$ is a bounded sequence. Since
$\X^*$ is also reflexive, from any subsequence
of $\{\mathcal E'(u_m)\}_{m\in \mathbb{N}}$, we can extract a further
subsequence $\{\mathcal E'(u_{m_k})\}_{k\in \mathbb{N}}$ that weakly-$*$
converges to an element $ \varphi \in
\X^*$. By using Lemma~\ref{NEW}(c),
we have for all $ z\in \Sc$,
$$
\vert \langle \mathcal E'(u_{m_k}), z\rangle \vert  \leq
C \Vert  u_{m_{k+1}}-u_{m_k} \Vert^2 \Vert  z\Vert.
$$
Taking the limit with $k \rightarrow \infty$, and using Lemma~\ref{NEW}(b), then $ \varphi \in
\X^*$ satisfies $\langle \varphi, z\rangle = 0$ for all $z\in \Sc.$
Since $\Sc$ is a universal dictionary, we conclude that
$\langle \varphi, z \rangle=0$ holds for all $z \in \X.$ Since from any subsequence of the
initial sequence $\{\mathcal E'(u_m)\}_{m\in\mathbb{N}}$ we can extract a
further subsequence that weakly-$*$ converges to the same limit
$0$, then the whole sequence converges to $0.$

\textbf{(e)} Let $s^*>1$ be such that $1/s^*+1/s =
1$. From Lemma~\ref{NEW}(b), we have $$
\sum_{m=1}^{\infty} \Vert u_{m}-u_{m-1} \Vert^s <\infty.$$ Since 
$\lim_{m\rightarrow \infty}\Vert u_{m}-u_{m-1} \Vert^s = 0$ we can extract
a subsequence $\{\Vert u_{m_k}-u_{m_{k-1}} \Vert^s\}_{k \ge 1}$ monotonically
decreasing and such that $$\sum_{k=1}^{\infty} \Vert u_{m_k}-u_{m_{k-1}} \Vert^s <\infty.$$
From Theorem 1 p.16 in \cite{Knopp},  $\lim_{k \rightarrow \infty} m_k \Vert
u_{m_k}-u_{m_{k-1}}\Vert^{s}\rightarrow 0.$ For $1<s\le 2$, we have
$2\le s^*$. Since $\lim_{k\to\infty} \Vert u_{m_k}-u_{m_{k-1}} \Vert
=0$, then $$\Vert u_{m_k}-u_{m_{k-1}}\Vert^{2s^*} \le \Vert
u_{m_k}-u_{m_{k-1}}\Vert^{2s}$$ for $k$ large enough, and therefore
 $\lim_{k \rightarrow \infty} m_k \Vert u_{m_k}-u_{m_{k-1}}\Vert^{2s^*} = 0.$

Since $u_{m_k} = \sum_{\ell=1}^k u_{m_\ell}-u_{m_{\ell-1}},$ we have
\begin{align*}
\vert\langle \mathcal E'(u_{m_k}),u_{m_k} \rangle \vert & \le  \sum_{\ell=1}^k
\vert \langle \mathcal E'(u_{m_k}),u_{m_\ell}-u_{m_{\ell-1}} \rangle \vert\\
&\le  C \sum_{\ell=1}^k \Vert u_{m_{k+1}}-u_{m_k}\Vert^2 \Vert u_{m_\ell}-u_{m_{\ell-1}} \Vert \quad
\text{(By Lemma~\ref{NEW}(c))}.
\end{align*}
By Holder's inequality, we have
\begin{align*}
\vert\langle \mathcal E'(u_{m_k}),u_{m_k} \rangle \vert  &\le  C \left( m_k \Vert u_{m_{k+1}}-u_{m_k}\Vert^{2s^*} \right)^{1/s^*} \left(
\sum_{\ell=1}^k \Vert u_{m_\ell}-u_{m_{\ell-1}} \Vert^s\right)^{1/s} \\ 
&\le  C \left( m_{k+1} \Vert u_{m_{k+1}}-u_{m_k}\Vert^{2s^*} \right)^{1/s^*} \left(
\sum_{\ell=1}^k \Vert u_{m_\ell}-u_{m_{\ell-1}} \Vert^s\right)^{1/s}.
\end{align*}
By taking limits as $k \rightarrow \infty$ in the above inequality we obtain the desired result. 
This ends the proof of proposition. \qed

\subsubsection*{Proof of Proposition~\ref{rate3}}
Recall that $\Omega_0$ is closed and bounded (see Section~\ref{ex1.5}). Then the set
$$
\Omega_0 -\Omega_0 =\{z \in \mathbb X: z=z_1-z_2 \text{ where } z_i \in \Omega_0 \text{ for } i=1,2\}
$$
is also bounded. Fix $c > 0$ be such that
$\Omega_0 -\Omega_0 \subset B_{c,\X}.$ Take $u \in \Omega_0$ and, by \eqref{eq:J_strong_convexity_consequence0},
$$
\mathcal E(u+z) - \mathcal E(u) \ge \langle \mathcal E'(u), z\rangle
$$
holds for all $z\in \X.$ Hence,
\begin{equation}\label{lemk1}
	\min_{z \in B_{\X,c}} \mathcal E(u+z) - \mathcal E(u)  \ge \min_{z \in B_{\X,c}}  \langle \mathcal E'(u), z\rangle.
\end{equation}
As $u \in \Omega_0,$ we have $u-\Omega_0 \subset \Omega_0 -\Omega_0 \subset  B_{\X,c},$ so that
$\Omega_0 \subset u+B_{\X,c}.$ Therefore,
$$
\mathcal E(u^*) = \min_{z \in X}\mathcal E(z) \le \min_{z \in B_{\X,c}} \mathcal E(u+z) \le \min_{z \in \Omega_0} \mathcal E(z) = \mathcal E(u^*).
$$
Thus, from \eqref{lemk1},
$$
\mathcal E(u^*)-\mathcal E(u) \ge \min_{z \in B_{\X,c}}  \langle \mathcal E'(u), z\rangle
$$
holds. Now, by using
$$
\min_{z \in B_{\X,c}}  \langle \mathcal E'(u), z\rangle = c \min_{z \in B_{\X,1}}  \langle \mathcal E'(u), z\rangle 
= - c \max_{z \in B_{\X,1}}  \langle \mathcal E'(u), z\rangle = - c \|\mathcal E'(u)\|_*,
$$
the proposition follows. \qed

\bigskip

\noindent \textbf{Acknowledgements:} Antonio Falc\'o would thank to Marie Billaud-Friess,  R\'egis Lebrun, Bertrand Michel, Anthony Nouy and Philipp Petersen their helpful discussions 
along the days in the Research School on High-dimensional Approximation and Deep Learning held in Nantes from May 16 to 20, 2022. 

\bigskip

\noindent \textbf{Funding :} P. M. Bern\'a was partially supported by the Grant PID2019-105599GB-I00 (Agencia Estatal de Investigación, Spain). A. Falc\'o was partially supported by  the grant number INDI21/15 from Universidad CEU Cardenal Herrera and by ESI International Chair@ CEU-UCH.

\bigskip

\noindent \textbf{Conflicts of Interest:} The authors declare no conflict of interest.

%
%
%


\begin{bibsection}
\begin{biblist}
\bib{Kantorovich}{book}{
title =     {Functional analysis},
   author =    {Akilov, G. P.},
   author = {Kantorovich, L. V.},
   publisher = {Pergamon Press},
    year =      {1982},
   edition =   {2d ed},
}

\bib{Ammar}{article}{
author = {Ammar A.}, 
author = {Huerta A.}, 
author = {Chinesta F.}, 
author = {Cueto E.}, 
author = {Leygue A.}, 
year = {2014}, 
title = {Parametric solutions involving geometry: A step towards efficient shape optimization}, 
journal = {Computer Methods in Applied Mechanics and Engineering}, 
volume = {268},
pages = {178–193}, 
doi = {10.1016/j.cma.2013.09.003},
}


\bib{BSU}{article}{
author = {Bachmayr, M.},
author = {Schneider, R.}, 
author = {Uschmajew, A.}, 
title = {Tensor Networks and Hierarchical Tensors for the Solution of High-Dimensional Partial Differential Equations}, 
journal = {Found. Comput. Math.}, 
volume = {16}, 
pages = {1423–1472}, 
year = {2016}, 
doi = {10.1007/s10208-016-9317-9}
}

\bib{BK}{article}{
author = {Bachmayr, M.}, 
author = {Kazeev, V.},
title = {Stability of Low-Rank Tensor Representations and Structured Multilevel Preconditioning for Elliptic PDEs}, 
journal = {Found. Comput. Math.}, 
volume = {20}, 
pages = {1175–1236}, 
year = {2020}, 
doi = {10.1007/s10208-020-09446-z},
}

\bib{Boucinha}{article}{
author = {Boucinha L.},
author = {Gravouil A.}, 
author = {Ammar A.},
title = {Space–time proper generalized decompositions for the resolution of transient elastodynamic models},
journal = {Computer Methods in Applied Mechanics and Engineering},
volume = {255},
pages = {67-88},
year = {2013},
doi = {10.1016/j.cma.2012.11.003},
}
\bib{CANDELA1}{article}{
author = {Candela, V.}, 
author = {Falc\' o, A.},
author = {Romero, P. D.},
title = {A general framework for a class of non-linear approximations with applications to image restoration},
journal = {Journal of Computational and Applied Mathematics},
volume = {330},
pages = {982-994},
year = {2018},
}

\bib{CANDELA2}{article}{
author = {Candela, V.}, 
author = {Romero, P. D.},
title = {Blind Deconvolution Models Regularized by Fractional Powers of the Laplacian},
journal = {J. Math. Imaging. Vis.},
volume = {32},
pages = {181–191},
year = {2008},
}

\bib{CAN05}{article}{
author = {Canuto, C.},
author = {Urban, K.},
title = {Adaptive optimization of convex functionals in Banach spaces},
journal = {SIAM J. Numer. Anal.}, 
volume ={42},
pages ={2043-2075},
year = {2005},
}

\bib{CIARLET}{book}{
author = {Ciarlet, P. G.},
title = {Introduction to Numerical Linear Algebra and Optimization},
series ={Cambridge Texts in applied Mathematics},
publisher = {Cambridge University Press},
date = {1989},
}

\bib{Coelho}{article}{
  doi = {10.1007/s00466-022-02214-6},
  year = {2022},
  author = {Coelho Lima I.}, 
  author = {Robens-Radermacher A.},  
  author = {Titscher T.}, 
  author = {Kadoke D.},
  author = {Koutsourelakis P-S.}, 
  author ={ Unger J. F.},
  title = {Bayesian inference for random field parameters with a goal-oriented 
  quality control of the {PGD} forward model's accuracy},
  journal = {Computational Mechanics}
}


\bib{Dinca}{article}{
author = {Dinca, G.},
author = {Jebelean, P.},
author = {Mawhin, J.},
title = {Variational and topological methods for Dirichlet problems with p-Laplacian},
journal = {Portugaliae Mathematica. Nova S\'erie}, 
volume ={58},
pages ={339-378},
year = {2001},
}


\bib{EKE99}{book}{
	author={Ekeland, I.},
		author={Teman, R.},
	title={Convex Analysis and Variational Problems},
	series={Classics in Applied Mathematics. SIAM},
	date={1999},
}


\bib{FHN}{article}{
author = {Falcó, A.}, 
author = {Hackbusch, W.}, 
author = {Nouy, A.}, 
title = {On the Dirac–Frenkel Variational Principle on Tensor Banach Spaces}, 
journal = {Found. Comput. Math.},
volume = {19}, 
pages = {159–204}, 
year = {2019}, 
doi = {10.1007/s10208-018-9381-4},
}

\bib{FalcoHackbusch}{article}{
author={Falc\'o, A.},
author={Hackbusch, W.},
title={On Minimal Subspaces in Tensor Representations},
journal={Found. Comp. Math.},
volume={12},
date={2012},
number={6},
pages={765--803},
issn={0022-1236},
review={\MR{3029146}},
doi={10.1016/j.jfa.2013.02.003},
}

\bib{FN}{article}{
	author={Falc\'o, A.},
	author={Nouy, A.},
	title={Proper generalized decomposition for nonlinear convex problems in tensor Banach spaces},
	journal={Number. Math.},
	volume={121},
	date={2012},
	number={},
	pages={503--530},
	issn={0022-1236},
	review={\MR{3029146}},
	doi={10.1016/j.jfa.2013.02.003},
}


\bib{Gonzalez}{article}{
  doi = {10.1016/j.matcom.2012.04.001},
  year = {2012},
  volume = {82},
  number = {9},
  pages = {1677--1695},
  author = {Gonz{\'{a}}lez D.},
  author = {Masson F.},
  author = {Poulhaon F.}, 
  author = {Leygue A.}, 
  author = {Cueto E.},   
  author = {Chinesta F.},
  title = {Proper Generalized Decomposition based dynamic data driven inverse identification},
  journal = {Mathematics and Computers in Simulation},
}

\bib{Greub}{book}{
	author = {Greub, W. H.}, 
	title = {Linear Algebra, 4th ed.},
	series ={Graduate Text in Mathematics}, 
	publisher ={Springer-Verlag}, 
	date={1981},
}


\bib{Knopp}{book}{
	author = {Knop, K.}, 
	title = {Infinite Sequences and Series}, 
	publisher ={Dover Publications, Inc.}, 
	date={1956},
}

\bib{Leygue}{article}{
author = {Leygue A.}, 
author = {Verron E.}, 
year = {2010}, 
title = {A first step towards the use of proper general decomposition method for structural optimization},
journal = {Arch. Computat. Methods. Eng.},
volume = {17},
pages = {465–472}, 
doi = {10.1007/ s11831-010-9052-3},
}

\bib{Niroomandi}{article}{
author = {Niroomandi S.}, 
author = {González D.}, 
author = {Alfaro I.}, 
author = {Bordeu F.}, 
author = {Leygue A.}, 
author = {Cueto E.}, 
author = {Chinesta F.}, 
year = {2013}, 
title = {Real-time simulation of biological soft tissues: A PGD approach},
journal = {Int. J. Numer. Meth. Biomed. Engng.},
volume = {29},
pages = {586–600},
doi = {10.1002/cnm.2544},
}



\bib{Mohri}{book}{
author = {Mohri, M.},  
author ={Rostamizadeh, A.}, 
author = {Talwalkar, A.},
title = {Foundations of Machine Learning},
series={Adaptive Computation and Machine Learning},
	edition ={Second Edition},
	publisher={The MIT Press},
	date={2018},
}



\bib{Petersen2}{article}{
	author ={Marcati, C.}, 
	author ={Opschoor, J.A.A.},
	author = {Petersen, P.C.},
	author = {Schwab, Ch.}, 
	title = {Exponential ReLU Neural Network Approximation Rates for Point and Edge Singularities}, 
	journal ={Found. Comput. Math.}, 
	year = {2022}, 
	doi = {10.1007/s10208-022-09565-9},
}

\bib{Petersen1}{article}{
	author = {Petersen, P.}, 
	author = {Raslan, M.}, 
	author = {Voigtlaender, F.},
	title = {Topological Properties of the Set of Functions Generated by Neural Networks of Fixed Size}, 
	journal = {Found. Comput. Math.}, 
	volume ={21}, 
	pages = {375–444},
	year = {2021}, 
}


\bib{Raissi1}{article}{
author = {Raissi, M.},
author = {Karniadakis, G.E.},
title = {Hidden physics models: machine learning of nonlinear partial differential equations},
journal = {J. Comput. Phys.},
volume  = {357}, 
pages ={125–141}, 
year = {2018},
doi ={10.1016/j.jcp.2017.11.039},
}

\bib{Raissi2}{article}{
author = {Raissi, M.},
author = {Perdikaris, P.},
author = {Karniadakis, G.E.},
title = {Physics-informed neural networks: a deep learning framework for solving forward and inverse problems involving nonlinear partial differential equations},
journal = {J. Comput. Phys.},
volume  = {378}, 
pages ={686–707}, 
year = {2019},
doi ={10.1016/j.jcp.2018.10.045},
}


\bib{Sawicki}{article}{
	author = {Sawicki, A.},
	author = {Karnas, K.},
	title = {Universality of Single-Qudit Gates},
	journal = {Ann. Henri Poincaré},
	volume  = {18}, 
	pages ={3515–3552}, 
	year = {2017},
	doi ={0.1007/s00023-017-0604-z},
}


\bib{Sheng}{article}{
author = {Sheng, H.},
author = {Yang, C.},
title = {PA penalty-free neural network method for solving a class of second-order boundary-value problems on complex geometries},
journal = {J. Comput. Phys.},
volume  = {428}, 
pages ={110085}, 
year = {2021},
doi ={10.1016/j.jcp.2020.110085},
}



\bib{Weinan1}{article}{
  author = {Weinan E.}, 
  author = {Bing Y.},
  doi = {10.1007/s40304-018-0127-z},
  year = {2018},
  publisher = {Springer Science and Business Media {LLC}},
  volume = {6},
  number = {1},
  pages = {1-12},
  title = {The Deep Ritz Method: A Deep Learning-Based Numerical Algorithm for Solving Variational Problems},
  journal = {Communications in Mathematics and Statistics},
}


\bib{Zeidler}{book}{
		author={Zeidler, E.},
	title={Nonlinear Functional Analysis and its Applications III. Variational Methods and Optimization.},
	series={CMS Books in Mathematics/Ouvrages de Math\'{e}matiques de la SMC},
	volume={26},
	publisher={Springer-Verlag},
	date={1985},
	}
\end{biblist}
\end{bibsection}
\end{document}